\documentclass[12pt]{amsart}

\usepackage{amssymb}
\usepackage{amsmath}
\usepackage{amscd}
\usepackage{float}
\usepackage{graphicx}
\usepackage{amsfonts}
\usepackage{pb-diagram}
\usepackage{color}

\newtheorem{cor}{Corollary}

\newtheorem{prop}{Proposition}
\newtheorem{defn}{Definition}
\newtheorem{rem}{Remark}

\begin{document}

\title{Modular framization of the BMW algebra}

\author{J. Juyumaya}
\address{Departamento de Matem\'aticas, Universidad de Valpara\'{\i}so \\
Gran Breta\~na 1091, Valpara\'{\i}so, Chile.}
\email{juyumaya@uvach.cl}

\author{S. Lambropoulou}
\address{ Departament of Mathematics,
National Technical University of Athens,
Zografou campus, GR-157 80 Athens, Greece.}
\email{sofia@math.ntua.gr}
\urladdr{http://www.math.ntua.gr/$\tilde{~}$sofia}

\thanks{Both authors were partially supported by Fondecyt 1085002, NTUA and  Dipuv.}




\date{}
\maketitle

\section{Introduction}

In this work we introduce  the concept of Modular Framization or simply Framization.  The  idea  originates from  the Yokonuma--Hecke algebra, shortened to Y--H algebra,  built from the classical Hecke algebra by framization. That is, the Y--H algebra is built by adding framing generators to the Hecke algebra and changing the Hecke algebra quadratic relation by a  new quadratic relation which involves the framing generators. Using then the Y--H algebras and a Markov trace constructed on them\cite{ju} we can  produce invariants for oriented framed knots\cite{jula1,jula3},  classical knots\cite{jula4} (which satisfies a cubic skein relation), and  singular knots\cite{jula2}. Moreover, we have connections of the Y--H algebras with virtual knots and transversal knot theory.

All these results invite us to try and apply the framization mechanism on other knot algebras, that is, algebras having  generators whose behavior involves braid relations and polynomial relations. For example, the Temperley--Lieb algebra, the BMW algebra, the singular Hecke algebra introduced by Paris and Rabenda, the  Hecke algebra of $B$ type, et cetera.

Let  $d$ and $n$ be  two positive integers.
The aim of this note  is to construct a framization $F_{d,n}$ of the Birman--Wenzl--Murakami algebra, also known as  BMW algebra,  and start a systematic study of this framization. We show that $F_{d,n}$ is finite dimensional and the \lq braid generators\rq\ of this algebra satisfy a
 quartic relation which is of minimal degree not containing the generators $t_i$. They also satisfy a quintic relation, as the smallest closed relation. We conjecture that the algebras $F_{d,n}$ support a Markov trace which allow to define  polynomial invariants for unoriented knots  in an analogous way that the Kauffman polynomial is derived from the BMW algebra.

in Section~3, we give some more knot algebras on which the framization can be applied, such as the Temperley--Lieb algebra, the singular Hecke algebra and   Hecke algebras of type $B$.

\section{The modular framization of BMW}

Let $K$ denote the field of rational functions ${\Bbb C}(l,m)$, where $l$ and $m$ are two unspecified parameters.  For any natural number $n$, J.~ Birman  and H.~Wenzl \cite{bw} and, simultaneously but independently, J.~Murakami \cite{mu}  defined  a unital associative $K$--algebra of two parameters, $C_n=C_n(l,m)$, which is known as the Birman--Wenzl--Murakami algebra or, simply, the BMW algebra. The algebra $ C_n$ with unity $1$ is defined by  two sets of generators, $g_1, \ldots , g_{n-1}$ and $h_1, \ldots , h_{n-1}$, satisfying: the  {\it braid relations} among the $g_i$'s:
$$
\text{(B1)} \qquad g_i g_{i+1} g_i = g_{i+1} g_i g_{i+1} \qquad \qquad \text{and}\qquad \qquad \text{(B2)} \qquad g_i g_j = g_j g_i \quad  \text{for}\quad \vert i-j\vert >1
$$
 together with the following relations:
\begin{equation}\label{bmw1}
g_i^2 = 1 -mg_i + ml^{-1}h_i
\end{equation}
\begin{equation}\label{bmw2}
g_ih_i =  l^{-1}h_i
\end{equation}
\begin{equation}\label{bmw3}
h_ig_{i\pm 1} h_i = lh_i
\end{equation}

From the defining relations of $C_n$ we deduce that the $g_i$'s are invertible and also the following important relations:
\begin{equation}\label{bmw4}
 g_i^{-1}  = g_i- m h_i + m
\end{equation}
\begin{equation}\label{bmw5}
h_ig_i = l^{-1}h_i
\end{equation}
\begin{equation}\label{bmw6}
h_ih_j = h_jh_i, \quad\text{for $\vert i-j\vert \geq 2$}
\end{equation}
\begin{equation}\label{bmw7}
h_i^2 = y h_i
\end{equation}
where
$$
y:= 1 +\frac{l^{-1}-l}{m}
$$
Other useful relations, which can be deduced from the above relations (see \cite{we1}), are:

\begin{equation}\label{bmw8}
h_i  h_{i\pm 1}h_i = h_i
\end{equation}
\begin{equation}\label{bmw9}
g_{i\pm 1}g_ih_{i\pm 1} = h_ig_{i\pm 1}g_i = h_ih_{i\pm 1}
\end{equation}
\begin{equation}\label{bmw10}
g_{i\pm 1}h_ig_{i\pm 1} = g_i^{-1}h_{i\pm 1}g_i^{-1}
\end{equation}
\begin{equation}\label{bmw12}
g_{i\pm 1}h_ih_{i\pm 1} = g_i^{-1}h_{i\pm 1}
\end{equation}
\begin{equation}\label{bmw13}
h_{i\pm 1}h_ig_{i\pm 1} = h_{i\pm 1}g_i^{-1}
\end{equation}

\begin{defn}\rm
Let $d$ be a natural number and let $y_0:= y$ and $y_1, \ldots ,y_{d-1}$ be unspecified parameters.  The {\it $d$--framization of the algebra $C_n$}, denoted $F_{d,n}= F_{d,n}(l,m, y_0,\ldots ,y_{d-1})$,  is defined as follows. The algebra $ F_{d,n}$ is the unital (with unity 1) associative algebra over $K$, defined through three sets of generators: the two sets of generators of the algebra $ C_n$ given above, together with \lq framing generators\rq\  generators $t_1, \ldots ,t_n$, satisfying all  defining relations of $ C_n$, except the quadratic relation in Eq. \ref{bmw1}, that is, (B1), (B2), (2), (3), together with the following relations:
\begin{equation}\label{fbmw1}
t_i^d = 1, \quad  t_i t_j= t_j t_i\qquad \text{for all $ i,j$}
\end{equation}
\begin{equation}\label{fbmw2}
t_i h_i = t_{i+1} h_i, \quad h_i t_i = h_i t_{i+1}
\end{equation}
\begin{equation}\label{fbmw3}
h_i t_i^k h_i = y_k h_i \quad 0\leq k\leq d-1
\end{equation}
\begin{equation}\label{fbmw4}
t_j g_i  =  g_i t_{s_i(j)} \qquad \text{for all $ i,j$}
\end{equation}
 where $s_i(j)$ is the effect  of  the transposition $s_i=(i, i+1)$ on $j$, and the quadratic relation
\begin{equation}\label{fbmw5}
g_i^2 = (1 -m) - me_i(g_i-1)  + ml^{-1}h_i
\end{equation}
 where
\begin{equation}\label{ei}
e_i :=\frac{1}{d}\sum_{s=0}^{d-1}t_i^s t_{i+1}^{-s}
\end{equation}
\end{defn}

\begin{rem}\rm
In the case $d=1$ we have $e_i=1$, hence $F_{d,n}$ coincides with $C_n$.
\end{rem}

\begin{rem}\rm
Mapping $t_i\mapsto 1$, $g_i\mapsto g_i$ and $h_i\mapsto h_i$ defines an algebra epimorphism from  the
algebra $F_{d,n}$ onto the algebra $C_n$.
\end{rem}

\begin{prop}\label{ginvrs}
For all $i$ we have:
\begin{enumerate}
\item $e_i^2= e_i$
\item $e_i h_i = h_ie_i = h_i$
\item The elements $g_i$ are invertible and
$$
g_i^{-1} = \frac{1}{1-m}g_i - \frac{m}{1-m}g_ie_i - mh_i + me_i
$$
\end{enumerate}
\end{prop}
\begin{proof}
Claim (1) is easy to prove, see \cite{jula1}.
Claim (2) comes from Eqs. \ref{fbmw2}. Finally, using (1) and (2) of the proposition and Eq. \ref{fbmw5}, Eq. \ref{bmw2} and Eq. \ref{bmw5}, one can verify that $g_i^{-1}g_i= g_ig_i^{-1}=1$. Thus (3) is also proved.
\end{proof}

\begin{prop}\label{quintic}
The elements $g_i$  satisfy  the quartic relation
\begin{equation}\label{eqproof3}
 g_i^4 +mg_i^3 +(m-2)g_i^2+m (m-1)g_i - (m-1) = ml^{-1}\left(m+l^{-2}-1\right)h_i
\end{equation}
and this is of minimal degree not containing the generators $t_i$. Also, they satisfy the
\lq closed\rq  quintic equation,
\begin{equation}\label{quintic}
\left(x-l^{-1}\right)\left(x^4 +mx^3 + (m-2)x^2 +m (m-1)x -(m-1)\right)= 0
\end{equation}
and this is of minimal degree not containing the generators $t_i$ and $h_i$.
Notice that

\noindent $x^4 +mx^3 + (m-2)x^2 +m (m-1)x -(m-1)= (x^2+mx-1)(x^2+ m-1)$.
\end{prop}
\begin{proof}
Multiplying  Eq. \ref{fbmw5} by $e_i$ and solving with respect to $me_i(g_i-1)$, we obtain
\begin{equation}\label{eqproof1}
me_i(g_i-1) =(1-m)e_i -e_i g_i^2 + ml^{-1}h_i
\end{equation}
Also, solving $me_i(g_i-1)$ directly from Eq. \ref{fbmw5} we have
$
me_i(g_i-1) =(1-m) - g_i^2 + ml^{-1}h_i.
$
Hence
\begin{equation}\label{eqproof2}
e_i(g_i^2 + m-1)=(g_i^2 + m-1)
\end{equation}
Multiplying now Eq. \ref{fbmw5} by $g_i$ we have
$g_i^3 = (1-m)g_i - me_ig_i^2+me_ig_i+ ml^{-2}h_i$. Then
$$
g_i^3
(1-m)g_i - me_i(g_i^2 + m -1)+m^2e_i +me_i(g_i-1)+ ml^{-2}h_i
$$
so, from Eq. \ref{eqproof2} we obtain
$$
g_i^3=(1-m)g_i - m(g_i^2 + m -1)+m^2e_i +me_i(g_i-1)+ ml^{-2}h_i
$$
Replacing in this last equation the expression in Eq. \ref{eqproof1} for $me_i(g_i-1)$, we obtain
$$
g_i^3 =-(m+1)g_i^2 + (1-m)g_i + m(l^{-1}+l^{-2})h_i + (1-m^2) + m^2e_i
$$
Hence,
\begin{equation*}
me_i =\frac{1}{m}\left[g_i^3+(m+1)g_i^2 + (m-1)g_i - m(l^{-1}+l^{-2})h_i + (m^2-1) \right]
\end{equation*}
Then, replacing this last expression for $me_i$ in Eq. \ref{fbmw5}, we obtain Eq. \ref{eqproof3}.
Finally, Eq. \ref{bmw5}  says $h_i(g_i-l^{-1})=0$, and substituting $h_i$ by Eq. \ref{eqproof3} we obtain Eq. \ref{quintic}.
\end{proof}

\subsection{\it Topological interpretations}

Note that the generators $t_1, \ldots, t_n$ together with the relations in  Eq. \ref{fbmw1} form a copy of the abelian group $({\Bbb Z}/d{\Bbb Z})^n$ and recall\cite{jula1,jula2,jula3,jula4} that the modular framed braid group ${\mathcal F}_{d,n}$ is given as ${\mathcal F}_{d,n}=({\Bbb Z}/d{\Bbb Z})^n \rtimes  B_n$, where the action of $B_n$ on $({\Bbb Z}/d{\Bbb Z})^n$ is given by the permutation induced by a braid on the indices. Geometrically, elements of ${\mathcal F}_{d,n}$ are classical braids on $n$ strands with an integer modulo $d$, the framing, attached to each strand. In particular, the framing generator $t_i$ stands for the identity framed braid with framing 1 on the $i$th strand and 0 elsewhere. Looking at  our baby-example, the Y-H algebra, it is shown in \cite{jula1} that it can be defined as a quotient of a  group algebra of ${\mathcal F}_{d,n}$ over an ideal generated by the quadratic relations in the Y-H algebra.

Passing now to the BMW algebra, it is a quotient of the classical braid group $B_n$. To see this consider the algebra $C_n$  generated by the $g_i$'s only, and view Eq. \ref{bmw4} as the defining relations for the $h_i$'s. Analogously, its $d$--framization $F_{d,n}$ contains a copy of the abelian group $({\Bbb Z}/d{\Bbb Z})^n$, generated by $t_1, \ldots, t_n$. So, if we exempt the $h_i$'s from the set of generators for the algebra $F_{d,n}$ and if we consider (3) of Proposition~\ref{ginvrs} as the defining relation for the $h_i$'s, we conclude  that $F_{d,n}$ can be seen as a quotient of an appropriate group algebra of ${\mathcal F}_{d,n}$. Further, the element $h_i$ can be seen represented in the category of $(n,n)$--tangles  as the elementary tangle consisting in two curved parallel horizontal arcs joining the endpoints $i$ and $i+1$ at  the top and at the bottom of the otherwise identity tangle. So, elements in the algebra $F_{d,n}$ can be viewed as framed $(n,n)$--tangles, with framings modulo $d$.

\begin{figure}[H]
\begin{center}
\begin{picture}(220,60)
\put(-1,55){\tiny{$0$}}
\put(22,55){\tiny{$0$}}
\put(31,55){\tiny{$1$}}
\put(57,55){\tiny{$0$}}
\put(64,55){\tiny{$0$}}
\put(88,55){\tiny{$0$}}
\qbezier(0,0)(0,30)(0,50)
\put(5,25){$\ldots$}
\qbezier(25,0)(25,30)(25,50)
\qbezier(33 ,50)(45,33)(58,50)
\qbezier(33 ,0)(45,17)(58,0)
\qbezier(66,0)(66,30)(66,50)
\put(71,25){$\ldots$}
\qbezier(91,0)(91,30)(91,50)
\put(100, 25){$=$}
\put(119,55){\tiny{$0$}}
\put(142,55){\tiny{$0$}}
\put(151,55){\tiny{$0$}}
\put(177,55){\tiny{$1$}}
\put(184,55){\tiny{$0$}}
\put(208,55){\tiny{$0$}}
\qbezier(120,0)(120,30)(120,50)
\put(125,25){$\ldots$}
\qbezier(145,0)(145,30)(145,50)
\qbezier(153 ,50)(165,33)(178,50)
\qbezier(153 ,0)(165,17)(178,0)
\qbezier(186,0)(186,30)(186,50)
\put(191,25){$\ldots$}
\qbezier(211,0)(211,30)(211,50)
\end{picture}
\caption{The relation $t_ih_i= t_{i+1}h_i$}\label{fig1}
\end{center}
\end{figure}

\begin{figure}[H]
\begin{center}
\begin{picture}(220,60)
\put(-1,55){\tiny{$0$}}
\put(22,55){\tiny{$0$}}
\put(31,55){\tiny{$0$}}
\put(57,55){\tiny{$0$}}
\put(64,55){\tiny{$0$}}
\put(88,55){\tiny{$0$}}
\qbezier(0,0)(0,30)(0,50)
\put(5,25){$\ldots$}
\qbezier(25,0)(25,30)(25,50)
\put(45,25){\circle{20}}
\qbezier(33 ,50)(45,33)(58,50)
\qbezier(33 ,0)(45,17)(58,0)
\qbezier(66,0)(66,30)(66,50)
\put(71,25){$\ldots$}
\put(30,27){\tiny{$k$}}
\qbezier(91,0)(91,30)(91,50)
\put(100, 25){$=$}
\put(115,25){\tiny{$y_k$}}
\put(129,55){\tiny{$0$}}
\put(152,55){\tiny{$0$}}
\put(161,55){\tiny{$0$}}
\put(187,55){\tiny{$0$}}
\put(194,55){\tiny{$0$}}
\put(218,55){\tiny{$0$}}
\qbezier(130,0)(130,30)(130,50)
\put(135,25){$\ldots$}
\qbezier(155,0)(155,30)(155,50)
\qbezier(163 ,50)(175,33)(188,50)
\qbezier(163 ,0)(175,17)(188,0)
\qbezier(196,0)(196,30)(196,50)
\put(201,25){$\ldots$}
\qbezier(221,0)(221,30)(221,50)
\end{picture}
\caption{The relation $h_it_i^kh_i= y_kh_i$}\label{fig1}
\end{center}
\end{figure}

 The elements $e_i$ seen  as elements of ${\Bbb C}{\mathcal
F}_{d,n}$   can be interpreted geometrically as the average of the sum of $d$ identity framed braids with framings as shown below for $e_1$.

\begin{figure}[H]
\begin{picture}(320,60)

\put(0,38){$e_1 =$}
\put(37, 41){$1$}
\qbezier(36,40)(41,40)(46,40)
\put(36,30){$d$}
\qbezier(55,20)(45,40)(55,60) 
\qbezier(65,20)(65,40)(65,60)
\qbezier(80,20)(80,40)(80,60)
\qbezier(95,20)(95,40)(95,60)
\put(105,40){$+$}
\qbezier(125,20)(125,40)(125,60)
\qbezier(140,20)(140,40)(140,60)
\qbezier(155,20)(155,40)(155,60)
\put(165,40){$+$}
\qbezier(185,20)(185,40)(185,60)
\qbezier(200,20)(200,40)(200,60)
\qbezier(215,20)(215,40)(215,60)
\put(225,40){$+$}
\put(240,40){$\cdots$}
\put(260,40){$+$}

\qbezier(280,20)(280,40)(280,60)
\qbezier(295,20)(295,40)(295,60)
\qbezier(310,20)(310,40)(310,60)
\qbezier(320,20)(330,48)(320,60) 
\put(62,65){\tiny{$0$}}
\put(77,65){\tiny{$0$}}
\put(92,65){\tiny{$0$}}
\put(122,65){\tiny{$1$}}
\put(130,65){\tiny{$d-1$}}
\put(153,65){\tiny{$0$}}
\put(182,65){\tiny{$1$}}
\put(190,65){\tiny{$d-2$}}
\put(213,65){\tiny{$0$}}
\put(270,65){\tiny{$d-1$}}
\put(292,65){\tiny{$1$}}
\put(307,65){\tiny{$0$}}

\end{picture}
\caption{The element $e_1 \in {\Bbb C} {\mathcal F}_{d,3}$} \label{fig3}
\end{figure}

\subsection{\it $F_{d,n}$ is finite dimensional}

\begin{prop}
Any element in $F_{d,n}$ can be written as a  $K$--linear combination of monomials of the form $\alpha f\beta$, where $\alpha$ and $\beta$ are monomials in $1, g_1,\ldots, g_{n-2}$, $h_1, \ldots , h_{n-2}$, $t_1, \ldots ,t_{n-1}$ and $f\in X_n:=\{t_n^s, g_{n-1}, t_{n-1}^sh_{n-1}t_{n-1}^r\,;\, 1\leq r,s\leq d\}$.
\end{prop}
\begin{proof}
The proof is by induction on $n$. For $n=2$ the proposition follows directly from the defining relations of $F_{d,2}$, that is, from   Eq. \ref{bmw2}, Eq. \ref{bmw3} and Eqs. \ref{fbmw1}--\ref{fbmw5}. We assume now  truth of the proposition for $2<k\leq n-1$. An arbitrary element in  $F_{d,n}$ is a $K$--linear combination of monomials $M$ in $1, g_1,\ldots, g_{n-1}$, $h_1, \ldots , h_{n-1}$, $t_1, \ldots ,t_n$. We shall check first the case where $M$ is such a monomial containing two elements $f_1, f_2$ of $X_n$. Further, by the induction hypothesis, we may assume that $M$ is in the form
$$
M= M_1f_1M_2f_2M_3
$$
where $M_i$ are monomials in $1, g_1,\ldots, g_{n-2}$, $h_1, \ldots , h_{n-2}$, $t_1, \ldots ,t_{n-1}$, each with at most one $f\in X_{n-1}$. We have two cases, according to whether $M_2$ contains or not an element  in $X_{n-1}$. In the case where $M_2$ does not have such an element, we have:
$$
M = M_1M_2f_1f_2M_3
$$
So, we must reduce $f_1f_2$ as a linear combination of monomials conforming with the statement of the proposition. We have nine  cases to reduce, but we shall consider the more representative ones, that
is the cases: $f_1=  g_{n-1}$ and  $f_2=t_{n-1}^sh_{n-1}t_{n-1}^r$. Applying  Eqs. (\ref{fbmw4}),  (\ref{bmw5}) and (\ref{fbmw2}), respectively, we obtain that $f_1f_2$ can be reduced:
$$
f_1 f_2 =  t_{n}^sg_{n-1}h_{n-1}t_{n-1}^r = l^{-1} t_{n}^sh_{n-1}t_{n-1}^r = l^{-1} t_{n-1}^sh_{n-1}t_{n-1}^r
$$

Suppose now that $M_2$ contains one element $1\not= f\in X_{n-1}$. So, we can  assume that $M$ is of the form $M= M^{\prime}f_1ff_2M^{\prime\prime}$, where $M^{\prime}, M^{\prime\prime}\in F_{d,n-1}$ and
$f_1, f_2\in X_{n}$. Hence, it is enough to show that $f_1ff_2$ is a linear  combinations of monomials as in the statement of the proposition; we shall show the reduction for only one case, as the rest of the 26 cases follow in the same way. Suppose  $f_1=f_2= g_{n-1}$ and $f=t_{n-2}^sh_{n-2}t_{n-2}^r$. We have:
\begin{eqnarray*}
f_1ff_2  =  g_{n-1}t_{n-2}^sh_{n-2}t_{n-2}^rg_{n-1}
& = &
t_{n-2}^sg_{n-1}h_{n-2}g_{n-2}t_{n-1}^r\qquad (\text{from Eq.  \ref{fbmw4}})\\
& = &
t_{n-2}^sg_{n-2}^{-1}h_{n-1}g_{n-2}^{-1}t_{n-2}^r\qquad (\text{from Eq. \ref{bmw10}})\\
& = &
g_{n-2}^{-1}t_{n-1}^sh_{n-1}t_{n-1}^rg_{n-2}^{-1}\qquad (\text{from Eq. \ref{fbmw4}})
\end{eqnarray*}
Using the expression of $g_{n-2}^{-1}$ in Proposition \ref{ginvrs}, the reduction of $f_1ff_2$ follows.
\end{proof}
\begin{cor}
$F_{d,n}$ is finite dimensional.
\end{cor}

\section{Speculations}

Apart from the framization of the BMW algebra, defined in Section~2,  we give below some more algebras on which the framization can be applied, recalling first the Yokonuma--Hecke algebra, which is the framization of the classical Hecke algebra of type $A$.

\bigbreak

\noindent {\it The modular framization of the Iwahori--Hecke algebra.} This is the Yokonuma--Hecke algebra, shortened to Y--H algebra. We recall the definition of the Y--H algebras. Fix $u \in {\Bbb C}\backslash \{0,1\}$. Given two  positive integers $d$ and $n$, we denote ${\rm Y}_{d,n} =  {\rm Y}_{d,n}(u)$ the Yokonuma--Hecke algebra, which is a unital associative algebra over ${\Bbb C}$,  defined by the generators
$1, g_1, \ldots, g_{n-1}, t_1, \ldots, t_{n}$, satisfying: the  braid relations (B1), (B2) for the $g_i$'s, relations (\ref{fbmw1}), (\ref{fbmw4}) for the $t_i$'s, together with the extra {\it Yokonuma quadratic relations}:
\begin{equation}\label{yhquadr}
g_i^2 = 1 + (u-1)e_i(1-g_i)
\end{equation}
 where $e_i$ as given in Eq. \ref{ei}. In the case $d=1$, ${\rm Y}_{d,n}$  coincides with the classical Iwahori--Hecke algebra of type $A$, $H_n(q)$. That is, the Y--H algebra is built by adding framing generators to $H_n(q)$ and changing the {\it Hecke quadratic relations}:
\begin{equation}\label{hecke}
g_i^2 = q\cdot 1 + (q-1)g_i
\end{equation}
by the quadratic relations Eq. \ref{yhquadr}, which involve the framing generators. Alternatively, as shown in \cite{jula1}, the Y--H algebra can be defined as a quotient of a  group algebra of the modular framed braid group ${\mathcal F}_{d,n}$ over an ideal generated by the quadratic relations Eq. \ref{yhquadr}. 
We note that, if we considered the quotient of  a  group algebra of the classical or the modular framed braid group over the quadratic relations Eq. \ref{hecke} we would obtain nothing more interesting, comparing to the much richer structure of the Y--H algebra.

Using  the Y--H algebras and a Markov trace constructed on them\cite{ju} we have constructed invariants for oriented framed knots\cite{jula1,jula3},  classical knots\cite{jula4}  and  singular knots\cite{jula2}. In the Y--H algebra the following closed cubic relation is satisfied, not involving the framing generators, which gives rise to a cubic skein relation for the invariant of classical knots. 
\begin{equation}\label{cubic1}
g_i^3 = -u g_i^2 + g_i +u
\end{equation}

\bigbreak

\noindent {\it The modular framization of the Temperley--Lieb algebra.} In \cite{jugoula} we define the framization of the classical Temperley--Lieb algebra, ${\rm TL}_{d,n}(u)$, the Yokonuma--Temperley--Lieb algebra, ${\rm YTL}_{d,n}(u)$, as follows.

\begin{defn}\rm The Yokonuma--Temperley--Lieb algebra, ${\rm YTL}_{d,n}(u)$, is defined as the following quotient of the Yokonuma-Hecke algebra:

$$
{\rm YTL}_{d,n}(u) = \frac{ {\rm Y}_{d,n}(u)}{ \langle g_ig_jg_i + g_ig_j + g_jg_i +g_i +g_j +1,\quad |i-j|=1  \rangle}
$$
\end{defn}

In \cite{jugoula} we also find appropriate inductive bases for the Yokonuma--Temperley--Lieb algebras and we construct a Markov trace on them. Using this trace we define topological invariants for various types of knots, in analogy to the case of the Y--H algebras.

\bigbreak

\noindent {\it The modular framization of the singular Hecke algebra.} A definition of the singular Hecke algebra, denoted $SH_n(q)$, was proposed by Paris and Rabenda \cite{para}. This algebra is defined  as the quotient of the group algebra of the singular  braid monoid  $SB_n$ over the Hecke algebra quadratic relations Eq. \ref{hecke}. Recall that the  monoid $SB_n$ is defined by:  the unit $1$,  the classical elementary braids $\sigma_i$ with their inverses $\sigma_i^{-1}$, $1\leq i\leq n-1$, which are subject to the braid relations (B1), (B2), and by the elementary singular braids $\tau_i$, $1\leq i\leq n-1$, together with the following relations:
\begin{equation}\label{hesin}
\begin{array}{rll}
\lbrack\sigma_i, \tau_j\rbrack
=
 \lbrack\tau_i, \tau_j\rbrack & =  &0 \quad \text{for}\quad \vert i-j\vert  >1 \\
   \lbrack\sigma_i, \tau_i \rbrack
& = &
0   \quad \text{for all}\quad i \\
 \sigma_i \sigma_j \tau_i
 &  = &
   \tau_j \sigma_i \sigma_j  \quad \text{for}\quad \vert i-j \vert = 1
\end{array}
\end{equation}
Keeping the same notation for $\tau_i$ in the quotient $SH_n(q)$ and corresponding $g_i$ to $\sigma_i$, the  singular Hecke algebra $SH_n(q)$ is the complex associative unital algebra defined by the generators $1, g_1\ldots g_{n-1}$,  $\tau_1, \ldots \tau_{n-1}$ with the
relations Eq. \ref{hecke} together with the relations in Eq. \ref{hesin},  placing $g_i$ instead of $\sigma_i$.

\begin{defn}\rm
Let $d$ be a natural number.  The {\it $d$--framization of the algebra $HS_n(u)$}, denoted $FS_{d,n}= FS_{d,n}(u)$, is defined as follows. The algebra $ FS_{d,n}$ is the unital (with unity 1) associative algebra over $\Bbb C$, defined through three sets of generators: the two sets of generators of the algebra $ SH_n(q)$ given above, together with generators $t_1, \ldots ,t_n$, satisfying all  defining relations of $SH_n(u)$, except the quadratic relation in Eq. \ref{hecke}, together with the relations of Eqs. \ref{fbmw1}, (\ref{fbmw4}) and the Yokonuma quadratic relations in Eqs. \ref{yhquadr}.
\end{defn}

\noindent {\it The modular framization of  $B$--type Hecke algebras.} Recall that the  Artin braid group of type $ B$, which we denote by $B_{1,n}$, is defined  by generators $T, \sigma_1, \ldots, \sigma_{n-1}$, satisfying the  braid relations (B1), (B2) for the $\sigma_i$'s and the following {\it $B$--type relations}:
\begin{equation}\label{brels}
\sigma_1T\sigma_1T  =  T\sigma_1T\sigma_1, \quad  \sigma_i T  =  T \sigma_i \qquad \text{if  \ $i>1$}
\end{equation}
For $q,Q \in {\Bbb C}\backslash \{0,1\}$, the classical Iwahori--Hecke algebra of type $B$, $ H_n(q,Q)$, can be seen as a
quotient of the group algebra  $\Bbb C B_{1,n}$  by factoring
out the ideal generated by Eqs. \ref{hecke},  where $g_i$ denotes the image of $\sigma_i$ in  $ H_n(q,Q)$, and the  relation:
\begin{equation}\label{btype}
 T^2=(Q-1)T+Q
\end{equation}
Further, for $q,u_1, \ldots, u_r \in {\Bbb C}\backslash \{0,1\}$, the cyclotomic Hecke algebra of type $B$ and of degree $r$, $H_n(q,r)$, can be defined as the quotient of the group algebra  $\Bbb C B_{1,n}$  by factoring
out the ideal generated by the  relations Eqs. \ref{hecke}  and the {\it cyclotomic relation}:
\begin{equation}\label{cyclot}
(T-u_1)(T-u_2)\cdots(T-u_r)=0
\end{equation}
Finally, the generalized Hecke algebra of type $B $, $H_n(q,\infty)$, is defined\cite{la2} as the quotient of the group algebra  $\Bbb C B_{1,n}$  over the relations Eqs. \ref{hecke} only. In \cite{la1,la2} Markov traces are constructed on all these algebras, giving rise to Jones--type invariants of knots in the solid torus.

\begin{defn}\rm
The {\it framed braid group of type} $B$, ${\mathcal F}_{1,n}$, is defined as ${\mathcal F}_{1,n}={\Bbb Z}^n \rtimes B_{1,n}$, where the action of $B_{1,n}$ on ${\Bbb Z}^n$ is given by the permutation induced by a braid of type $B$ on the indices. Geometrically, elements of ${\mathcal F}_{1,n}$ are  braids on $n+1$ strands with the first strand fixed and with an integer, the framing, attached to each one of the rest $n$ strands.
 The {\it modular framed braid group of type $B$}, ${\mathcal F}_{d,1,n}$, is defined similarly, except that $B_{1,n}$ acts on  $({\Bbb Z}/d{\Bbb Z})^n$ and the framings are modulo $d$. So, the group ${\mathcal F}_{d,1,n}$ is defined  by generators $T, \sigma_1, \ldots, \sigma_{n-1}, t_1,\ldots, t_n$, satisfying the  braid relations (B1), (B2) for the $\sigma_i$'s, relations  Eqs. \ref{brels} for the generator $T$ and  relations (\ref{fbmw1}), (\ref{fbmw4}) for the framing generators $t_i$'s.
\end{defn}

\begin{defn}\rm
Let $d$ be a natural number.  The {\it $d$--framization of the Iwahori--Hecke algebra of $B$--type} $H_n(q,Q)$, denoted $FH_{d,Q,n}= FH_{d,Q,n}(u)$, is defined as the quotient of the modular framed braid group of type $B$, ${\mathcal F}_{d,1,n}$, over the quadratic relations in Eq. \ref{btype} and Eqs. \ref{yhquadr}. The {\it $d$--framization of the cyclotomic Hecke algebra $H_n(q,r)$}, denoted $FH_{d,r,n}= FH_{d,r,n}(u)$, is defined as the quotient of the modular framed braid group of type $B$, ${\mathcal F}_{d,1,n}$, over the quadratic relations in Eq. \ref{cyclot} and Eqs. \ref{yhquadr}. Finally, the {\it $d$--framization of the generalized Hecke algebra of type $B$, ${\mathcal H}_n(q,\infty)$}, denoted $FH_{d,n}= FH_{d,n}(u)$, is defined as the quotient of the modular framed braid group of type $B$, ${\mathcal F}_{d,1,n}$, over the quadratic relations in Eqs. \ref{yhquadr}.
\end{defn}
$B$--type framizations are studied in \cite{jula5}, including the $B$--type Temperley--Lieb algebras and $B$--type BMW algebras.

\end{document}